\documentclass[12pt,a4paper]{amsart}
\usepackage{amsmath,amsthm,amssymb}

\usepackage[shortlabels]{enumitem}

\usepackage[totalwidth=15.75cm,totalheight=22.275cm]{geometry}

\usepackage{graphicx}

\usepackage{hyperref}
\usepackage[usenames, dvipsnames]{color}
\usepackage[flushmargin]{footmisc}

\usepackage{epigraph}

\numberwithin{equation}{section}

\makeatletter
\def\swappedhead#1#2#3{%
  \thmnumber{\@upn{\the\thm@headfont #2\@ifnotempty{#1}{.~}}}%
  \thmname{#1}%
  \thmnote{ {\the\thm@notefont(#3)}}}
\makeatother
\swapnumbers

\theoremstyle{plain}
\newtheorem{thm}{Theorem}[section]%
\newtheorem{prop}[thm]{Proposition}%

\theoremstyle{definition}
\newtheorem{question}[thm]{Question}%

\newcommand*{\defeq}{\mathrel{\vcenter{\baselineskip0.5ex \lineskiplimit0pt
                     \hbox{\scriptsize.}\hbox{\scriptsize.}}}%
                     =}
\newcommand{\eqdef}{=\mathrel{\vcenter{\baselineskip0.5ex \lineskiplimit0pt
                     \hbox{\scriptsize.}\hbox{\scriptsize.}}}}

\newcommand{\eps}{\varepsilon}
\renewcommand{\theta}{\vartheta}
\renewcommand{\phi}{\varphi}

\newcommand{\PP}{\mathcal{P}}

\newcommand{\C}{{\mathbb{C}}}
\newcommand{\Ch}{\hat{\C}}
\newcommand{\D}{{\mathbb{D}}}

\newcommand{\re}{\operatorname{Re}}
\newcommand{\dist}{\operatorname{dist}}
\newcommand{\sing}{\operatorname{sing}}

\title{Singular orbits and Baker domains}

\makeatletter
\@namedef{subjclassname@2020}{%
  \textup{2020} Mathematics Subject Classification}
\makeatother

\begin{document}

\author{Lasse Rempe} 
\address{Dept. of Mathematical Sciences \\
	 University of Liverpool \\
   Liverpool L69 7ZL\\
   UK \\ 
	 \textsc{\newline \indent 
	   \href{https://orcid.org/0000-0001-8032-8580%
	     }{\includegraphics[width=1em,height=1em]{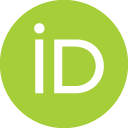} 
	        {\normalfont https://orcid.org/0000-0001-8032-8580}}
	       }}
\email{l.rempe@liverpool.ac.uk}

\begin{abstract}
  We show that there is a transcendental meromorphic function with an invariant Baker domain $U$ such that
    every singular value of $f$ is a super-attracting periodic point. This answers a question of Bergweiler from 
    1993. We also show that $U$ can be chosen to contain arbitrarily large round annuli, centred at zero, of definite modulus. 
    This answers a question of Mihaljevi\'c and the author from 2013, and complements recent 
    work of Bara\'nski et al concerning
    this question. 
\end{abstract}

\maketitle

\epigraph{I long for the Person from Porlock \\
To bring my thoughts to an end, \\
I am becoming impatient to see him\\
I think of him as a friend.}{Stevie Smith}

\section{Introduction}
    Let $f\colon\Ch\to\Ch$ be a rational function of degree at least $2$. The \emph{Fatou set} $F(f)$ consists of those points $z\in\C$ near which the 
     iterates of $f$ form a normal family in the sense of Montel. In other words, these are the points at which the dynamics generated by $f$ 
     is stable under small perturbations. A connected component of $F(f)$ is called a \emph{Fatou component}; such a component is 
     \emph{invariant} if $f(U)\subset U$. An invariant Fatou component may be an immediate basin of attraction of an attracting or parabolic fixed point, 
     or a simply or doubly connected domain on which $f$ is conjugate to an irrational rotation. 
     
   It was shown by Fatou \cite[\S 30--31]{fatoumemoire2} that there is a close relationship between invariant Fatou components and the critical values of $f$.
     Indeed, every attracting or parabolic basin must contain a critical value, and the boundary of any rotation domain is contained in the 
     \emph{postcritical set} $\PP(f)$, i.e., the closure of the union of all critical orbits of $f$. See Lemmas~8.5, 15.7
     and Theorems~10.15, 11.17 of  \cite{milnor}. 

   These relationships carry over to the case of a transcendental meromorphic function $f\colon \C\to\Ch$, with the set of critical values 
    replaced by the set $\sing(f^{-1})$ of critical or asymptotic values of $f$, and the postcritical set by the \emph{postsingular set} defined analogously. 
    In this setting, there is another possible type of invariant Fatou component: An \emph{invariant Baker domain}, in which the iterates 
    converge to the essential singularity at~$\infty$. 
    Bergweiler~\cite[Question~4]{waltersurvey} asked whether there is a relation between 
    $\sing(f^{-1})$ and the boundary of such a Baker domain. He also asked a more precise version of this
    question~\cite[Question~5]{waltersurvey}:
    \begin{question}\label{qu:question5}
     Is it possible that a meromorphic function $f$ has Baker domains if the forward orbit of $z$ is bounded for all
     $z\in \sing(f^{-1})$?
    \end{question}
    We give an affirmative answer.
    \begin{thm}[Baker domains and super-attracting points]\label{thm:question5}
      There is a transcendental meromorphic function $f\colon\C\to\Ch$ with a Baker domain such that
        every point of $\sing(f^{-1})$ is a super-attracting periodic point of period $2$. 
    \end{thm}

  Regarding~\cite[Question~4]{waltersurvey}, Bergweiler \cite[Theorem~3]{bergweilerinvariantdomains}
     obtained the following answer when $f$ is a transcendental \emph{entire} function: If 
     $U$ contains no critical or asymptotic values of $f$, then
     there exists a sequence $(p_n)_{n=0}^{\infty}$ of postsingular points $p_n\in \PP(f)$
     of $f$ such that 
       \begin{enumerate}[(a)]
        \item $\lvert p_n\rvert \to \infty$;
        \item $\frac{\dist(p_n,W)}{\lvert p_n\rvert} \to 0$; and
        \item $\lim_{n\to\infty} \lvert p_{n+1}\rvert / \lvert p_n\rvert = 1$.\label{item:bergweilerestimate}
        \end{enumerate}
    On the other hand~\cite[Theorem~1]{bergweilerinvariantdomains}, there is an entire function with a Baker domain $U$ such that $\dist( \PP(f), \partial U)>0$, where
    $\dist$ denotes Euclidean distance. 

   For Baker domains of general transcendental \emph{meromorphic} functions, it was shown in 
      \cite[Theorem~1.5]{lassehelenawandering} that an analogue of \cite[Theorem~3]{bergweilerinvariantdomains} 
      holds,
      without the hypothesis that $U$ contains no critical or asymptotic values, where condition~\ref{item:bergweilerestimate} is replaced by the weaker 
        \begin{equation}
            \limsup_{n\to\infty} \frac{\lvert p_{n+1}\rvert}{\lvert p_n\rvert} < \infty.\label{eqn:weakerestimate}
       \end{equation}
     The article~\cite{lassehelenawandering} also poses the following question. 
       \begin{question}[{\cite[Remark on p.~1603]{lassehelenawandering}}]\label{qu:baker}
         Can~\eqref{eqn:weakerestimate} be replaced by~\ref{item:bergweilerestimate}?
       \end{question}
       In \cite[Theorem~A]{baranskietalsingularities}, it is shown that the answer is positive if $\C\setminus U$ has an unbounded connected component. 
       Here we show that the answer to Question~\ref{qu:baker} is negative in general. 
      
       \begin{thm}[Annuli of definite modulus]\label{thm:question4}
         Let $\rho>0$. Then the function in Theorem~\ref{thm:question5} can be chosen such that there is 
         a sequence $R_j\to\infty$ with
           \begin{equation}\label{eqn:largeannuli} \{z\in\C\colon R_j<\lvert z\rvert < \rho R_j \}\subset U \end{equation}
           for all $j$.
       \end{thm} 

   The constructions in our paper are inspired by those of Rippon and Stallard~\cite[Theorem~1.2]{rippon_stallard_2006} and Bara\'nski et al~\cite[Theorem~C]{baranskietalabsorbing}. Both of these construct meromorphic functions
   with multiply-connected Baker domains, obtained from 
   an affine map by inserting  poles at a sequence of points tending to infinity. 
  We shall use quasiconformal surgery instead of explicit formulae, which allows us to obtain
   the required control of the singular orbits of $f$.  
       
\subsection*{Notation} $\C$, $\Ch$ and $\D$ denote the complex plane, Riemann sphere and unit disc, respectively. The Euclidean disc of radius $r$ around $z\in\C$ is denoted by
   $D(z,r)$.

\subsection*{Acknowledgements} I am grateful to Bogus{\l}awa Karpi\'nska, whose excellent presentation about~\cite{baranskietalsingularities} at the 
  2020 online conference \emph{On geometric complexity of Julia sets - II} inspired this work. I also thank her and
  Walter Bergweiler for interesting discussions following the talk, and impan and the Banach Centre in Warsaw for
  hosting the above conference. Finally, I am grateful to the referee for their thoughtful comments, which improved the paper.
       
\section{Quasiconformal surgery}\label{sec:surgery}

  The function in Theorem~\ref{thm:question5} is constructed by a ``cut and paste''
    quasiconformal surgery (see \cite[Chapter~7]{brannerfagella}), starting from the linear map
     \begin{equation}\label{eqn:f0} f_0(z) \defeq  \mu z, \end{equation}
     where $\mu>1$. 
      Consider the closed and connected set
         \[ U_0 \defeq \{z\in\C\colon \re z\geq 0\} \cup \bigcup_{j\geq 0} \left\{z\in\C\colon \mu^j \leq \lvert z\rvert \leq \mu^{j+\frac{1}{2}}\right\}, \]
       which is forward-invariant under $f_0$. 
       Also let $D_0$ be a closed disc
           \[ D_0 = \overline{D(\zeta_0,r_0)} \subset \{z\in\C\colon \re z < 0 \text{ and } \sqrt{\mu}<\lvert z\rvert < \mu \} \subset \C\setminus U_0; \] 
       If we define 
        \[ D_j \defeq f_0^j(D_0) = \overline{D(\mu^j \cdot \zeta_0 , \mu^j\cdot r_0)} \eqdef \overline{D(\zeta_j,r_j)},\]
         then all $D_j$ are disjoint from $U_0$, and from each other. 
          
    We construct our functions from $f_0$ by a quasiconformal surgery that inserts poles in the discs $D_j$. The following
      makes this procedure precise.
      \begin{prop}[Quasiconformal surgery of $f_0$]\label{prop:qcsurgery}
        Let $\mu>1$, and let $f_0$, $U_0$ and $(D_j)_{j=0}^{\infty}$ be defined as above.
                Let $K>1$, and let  
            \[ h_j\colon D_j \to \Ch \]
          be a sequence of $K$-quasiregular functions such that each $h_j$ extends continuously 
          to $\partial D_j$ with $h_j=f_0$ on $\partial D_j$. Then 
      \begin{equation}\label{eqn:Fdefn} F\colon \C\to\Ch; \quad z\mapsto \begin{cases} 
                          h_j(z) &\text{if } z\in D_j \\
                          f_0(z) &\text{if }z\notin \bigcup_{j} D_j. \end{cases} \end{equation}
         is $K$-quasiregular, with no finite asymptotic values, and every critical value of $F$ is a critical value of some  $h_j$. 
         
         Suppose furthermore that, for infinitely many $j$, the map $h_j$ is not a homeomorphism
           $h_j\colon D_j\to D_{j+1}$, and that for every $j$ there is an open set $X_j$ with the following properties.
           \begin{enumerate}[(a)]
             \item The dilatation of $h_j$ is supported on $X_j$.
             \item $F$ is conformal on $F^n(X_j)$ for all $n\geq 1$.\label{item:conformalonforwardimages}
           \end{enumerate}
           
          Then there is a quasiconformal homeomorphism $\psi\colon\C\to\C$ with $\psi(0)=0$ such that 
             \begin{equation}\label{eqn:fdefn} f\defeq \psi\circ F \circ \psi^{-1} \end{equation}
             is a transcendental meromorphic function. Moreover, $\psi$ is conformal on $U_0$, and
               $\psi(U_0)$ is contained in an invariant Baker domain $U$ of $f$. 
         \end{prop}
     \begin{proof}
       The claim that $F$ is $K$-quasiregular follows from Royden's glueing lemma~\cite[Lemma~2]{bersmoduli}, or the
        quasiconformal removability of quasiarcs~\cite[Theorem~1.19]{brannerfagella}. 
        Any curve to infinity must
        intersect $U_0$, and therefore $F$ is unbounded on any such curve. In particular, $F$ has no finite asymptotic values.
        As $F=f_0$ on $U_0$, every critical value of $F$ must come from one of the $h_j$. 
        
        Now suppose that the additional conditions hold. Since $h_j$ maps $\partial D_j$ to $\partial D_{j+1}$ in one-to-one
          fashion, if $h_j$ is not a homeomorphism, then $h_j$ must have at least one pole in $D_j$. As this happens
          for infinitely many $j$, we see that $F$ has infinitely many poles. Let $\omega_0\defeq F_0^*(0)$ be the complex dilatation
          of $F$. By~\ref{item:conformalonforwardimages}, the restriction of $\omega_0$ to 
            \[ X \defeq \bigcup_{j,n\geq 0} F^n(X_j) \]
 	is forward-invariant under $F$. 
           As $F$ is meromorphic outside of $X$, we may extend $\omega_0|_X$ by pull-back to the grand orbit of $X$. Extending the
           resulting differential by the standard complex structure, we obtain an $F$-invariant measurable Beltrami differential $\omega$ on all of $\Ch$. 
           Observe that $\omega \equiv 0$ on $U_0$. 
           
         By the measurable Riemann mapping theorem, there is a quasiconformal homeomorphism $\psi\colon\C\to\Ch$ that 
           solves the Beltrami equation for $\omega$. Then $\psi$ is conformal on $U_0$. Moreover, the map $f$ defined by~\eqref{eqn:fdefn} is meromorphic; it is transcendental
           since $F$ has infinitely many poles. Let $z\in \psi(U_0)$; say $z = \psi(\zeta)$ with $\zeta\in U_0$. Then
              \begin{align*}  f_0(z) &= \psi(F(\psi^{-1}(z))) = \psi(F(\zeta)) \in \psi(U_0), \quad\text{and} \\
                f_0^n(z) &=  \psi(F^n(\psi^{-1}(z))) = \psi(F^n(\zeta)) = \psi(f_0^n(\zeta)) \to \infty. \end{align*}
                So $\psi(U_0)$ is contained in a Baker domain
           of $f$, as claimed. 
     \end{proof}

    \begin{prop}[Large annuli in $U$]\label{prop:moduli}
     For every $\rho>1$, there is $\mu>1$ with the following property. Let $f_0(z)=\mu\cdot z$, and let $f$ be obtained 
     from $f_0$ as in Proposition~\ref{prop:qcsurgery}.
     Then $f$ satisfies~\eqref{eqn:largeannuli} for a sequence $R_j\to\infty$.
    \end{prop}
    \begin{proof}
      Recall that $U_0$ contains arbitrarily large round annuli of modulus at least $M \defeq (\log \mu)/2$. If
        $M\geq \Lambda(\rho)$, where $\Lambda(\rho)$ is the modulus of the Teichm\"uller 
        annulus~\cite[\S4--11]{ahlforsconformalinvariants}, then any annulus 
        that separates $0$ from $\infty$ and has modulus at least
       $M$ contains a round annulus of modulus $\log \rho$ centred at zero. 
       Hence we can take 
       $\mu = e^{2\Lambda(\rho)}$. 
    \end{proof}

  \section{Proof of Theorem~\ref{thm:question5}}
   To prove Theorem~\ref{thm:question5}, we construct a suitable sequence of functions $h_j$ to use in 
     Proposition~\ref{prop:qcsurgery}. The idea is that each $h_j$ has critical values in $D_{j+1}$, which are then
     mapped back to the original critical points by $h_{j+1}$. We begin with the following, where $f_0$ is again given by~\eqref{eqn:f0}.
             
    \begin{prop}\label{prop:qcexistence}
        Let $\Delta=D(\zeta,r) \subset \C$ be a round disc, and let $K>1$, $0<\eta<r$, and $0<\theta < \mu\cdot \eta$.
          For every $a\in\C\setminus f_0(\Delta)$, there 
           is a $K$-quasiregular map 
               \[ g\colon \Delta\to\Ch \]
              such that
              \begin{enumerate}[(1)]
                \item $g$ extends continuously to $\partial \Delta$, where it agrees with $f_0$.\label{item:continuousextension}
                \item $g(\zeta)=a$.\label{item:correctimageofcentre}
                \item $g$ has exactly two critical points, $c_1$ and $c_2$, with $g(c_1)\neq g(c_2)$.\label{item:twocriticalpoints}
                \item Each critical point $c_j$ satisfies $g(c_j)\in D(f_0(\zeta),\theta)$.\label{item:criticalvalues}
                \item $g$ is meromorphic on $D(\zeta,\eta)$.\label{item:supportofdilatation}
                \item $g$ is injective (and hence quasiconformal) on 
                   $\Delta\setminus D(\zeta,\eta)$, and 
                  $\mu r > \lvert g(z)-f_0(\zeta)\rvert > \theta$ for $r>\lvert z-\zeta\rvert \geq \eta$.\label{item:annulusimage}
               \end{enumerate}
    \end{prop}
    \begin{proof}
      Set \[ \alpha \defeq \frac{a}{r\mu} - \frac{\zeta}{r} \neq 0. \]
         Let $\eps>0$, and consider the map
             \[ \phi = \phi_{\alpha,\eps}\colon \D\to \Ch; \quad z\mapsto z + \frac{\eps \alpha}{\alpha z+\eps} = z + \frac{\eps}{z-p}, \] 
         where $p=-\eps/\alpha$. Then  $\phi(0)=\alpha$ and $\phi(p)=\infty$. By direct calculation,            
            $\phi$ has two simple critical points at 
             $p \pm \sqrt{\eps}$, with critical values at 
           $p\pm 2\sqrt{\eps}$.        
           
          Now define $A\colon\Ch\to\Ch$ to be the affine map taking $\D$ to $\Delta$; i.e., 
             $A(z) = r\cdot z + \zeta$. Define 
              \[   g_{\eps}\colon D(\zeta,\eta)\to \Ch; \quad z\mapsto
                                      \mu\cdot A\bigl(\phi_{\alpha,\eps}\bigl(A^{-1}(z)\bigr)\bigr).  \]
         For all sufficiently small $\eps$, 
            $g_{\eps}$ satisfies~\ref{item:twocriticalpoints},~\ref{item:criticalvalues} and~\ref{item:supportofdilatation}.
            Furthermore, 
             \[ g_{\eps}(\zeta) = \mu\cdot (r\alpha + \zeta) = a, \]
         so $g$ also satisfies~\ref{item:correctimageofcentre}. 
         
         Set $\Delta_{\eta} \defeq D(\zeta,\eta)$. 
         As $\eps\to 0$, $g_{\eps}\to f_0$ uniformly on a neighbourhood of $\partial\Delta_{\eta}$. In particular, for sufficiently
           small $\eps$, $g_{\eps}$ is conformal on a neighbourhood of $\partial \Delta_{\eta}$, and satisfies 
           \[ \mu r > \lvert g_{\eps}(z) - f_0(\zeta)\rvert > \theta \] 
           when $z\in\partial \Delta_{\eta}$. It follows that there is a quasiconformal homeomorphism 
             $G$ that maps the round annulus $\Delta \setminus \overline{\Delta_{\eta}}$ to the annulus
               bounded by $f_0(\partial\Delta)$ and $g_{\eps}(\partial \Delta_{\eta})$, and which
               agrees with $f_0$ on $\partial\Delta$ and with $g_{\eps}$ on $\partial \Delta_{\eta}$; see~\cite{lehtoannuli}. 
               We define $g$ to agree with $g_{\eps}$ on $D_{\eta}$ and with $G$ outside; then $g$
               satisfies~\ref{item:continuousextension}--\ref{item:annulusimage}. 

       Finally, we claim that the dilatation of $G$, and thus of $g$, tends to 1 as $\eps\to 0$. Indeed, this follows readily from the construction in~\cite{lehtoannuli}.
       Alternatively,
       for small $\eps$, we can define $G$ directly by linear interpolation on each radius of $\Delta \setminus \overline{\Delta_{\eta}}$. An elementary
       calculation shows that the dilatation of $G$ tends to $1$. 
       Hence, for sufficiently small $\eps$,
          $g$ is $K$-quasiconformal as claimed.
    \end{proof}

     For the remainder of the paper, let $D_j$ denote the discs introduced in Section~\ref{sec:surgery}. 
          
       \begin{prop}\label{prop:hjquestion5}
         For any $K>1$, there is a sequence of $K$-quasiregular maps
              \[ h_j \colon D_j\to\Ch \] with the following properties for all $j$.
            \begin{enumerate}[(a)]
              \item For every $j$, the map $h_j$ has at least one critical point.\label{item:notinjective} 
              \item If $c$ is a critical point of $h_j$, then $h_{j}(c)\in D_{j+1}$ and $h_{j+1}(h_j(c))=c$.
            \end{enumerate}
            Furthermore, for each $j$ there are open sets $V_j\subset W_j\subset D_j$ such that
            \begin{enumerate}[(a),resume]
              \item $D_j\setminus V_j$ is compact;\label{item:compactcomplement}
              \item $h_j$ is injective on $W_j$;\label{item:injectiveonWj}
              \item $h_j(z)=f_0(z)$ for $z\in V_j$;\label{item:maponVj}
              \item $h_j(W_{j})\subset V_{j+1}$;\label{item:WjtoVj}
              \item the dilatation of $h_j$ is supported on $W_j$.\label{item:dilatationsupport} 
            \end{enumerate}
       \end{prop}
       \begin{proof}
         The sequence of functions is constructed inductively, in such a way as to ensure the following
          inductive hypotheses. 
           For $j\geq 0$, let $\mathcal{C}_j\subset D_j$ be the set of critical points of $h_{j}$ and 
             $\Omega_{j+1}=h_{j}(\mathcal{C}_j)\subset \C$ the set of critical values. 
             The construction will ensure the following inductive hypotheses for $j\geq 0$:
             \begin{enumerate}[(A)]
               \item $\Omega_{j+1}\subset D_{j+1}$.  \label{item:omegainDj}
               \item $h_j\colon {\mathcal{C}_j}\to\Omega_{j+1}$ is bijective.\label{item:injective}
               \item $h_j(W_j)\subset D_{j+1}$ and $\overline{h_j(W_j)}\cap \Omega_{j+1} = \emptyset$. \label{item:criticalvaluesaway}
              \end{enumerate} 
             To anchor the recursive construction, we also set 
               $\Omega_0 \defeq \{\zeta_0\}$. This means that~\ref{item:omegainDj}
               holds also for $j=-1$.  
         
         Now suppose that $j\geq 0$ is such that $h_i$ has been constructed for $i<j$, satisfying the inductive hypotheses. 
          To define $h_j$, let $\omega_j^1,\dots,\omega_j^m$ be the elements of $\Omega_{j}$. Choose pairwise disjoint
            small discs $\Delta_j^i=D(\omega_j^i,\rho_j^i)\Subset D_j$ centred at the 
             $\omega_j^i$, for $j=1,\dots,m$; this is possible by~\ref{item:omegainDj}. If $j>0$, then by~\ref{item:criticalvaluesaway} we may
             also ensure that  the closures of the $\Delta_j^i$ do not intersect $\overline{h_{j-1}(W_{j-1})}$. 
         Set               
             \begin{equation}\label{eqn:Vjdefn}
               V_j \defeq D_j\setminus \bigcup_{i=1}^m \overline{\Delta^i_j}\supset h_{j-1}(W_{j-1}). \end{equation}    
              
          For each $i=1,\dots,m$, apply Proposition~\ref{prop:qcexistence}
           to $\Delta=\Delta_j^i$, where $a = c_{j-1}^i \in \mathcal{C}_{j-1}$ is such that 
             \begin{equation}\label{eqn:cji} h_{j-1}(c_{j-1}^i)=\omega_j^i,\end{equation}
            if $j>0$, and $c_{-1}^1=0$ otherwise. We may choose $\theta$ and $\eta=\eta^i_j$ 
           arbitrary, subject to the conditions
           given in the proposition. Let $g^i_j$ denote the maps obtained. We define 
             \begin{align}\label{eqn:hjdefn} h_j(z) &\defeq \begin{cases}
                                      f_0(z) &\text{if }z\in V_j;\\
                                      g^i_j(z) &\text{if }z\in \overline{\Delta^i_j}\end{cases} \qquad\text{and}\\
             W_j &\defeq D_j \setminus \bigcup_{i=1}^m \overline{D(\omega_j^i,\eta_j^i)} \supset V_j.\label{eqn:Wjdefn} \end{align}
        We claim  that $h_j$ satisfies~\ref{item:omegainDj}--\ref{item:criticalvaluesaway}.
        For the remainder of the proof, references to \ref{item:continuousextension}--\ref{item:supportofdilatation} shall mean the 
        corresponding properties of the $g^i_j$ established in Proposition~\ref{prop:qcexistence}. 

        The critical points and critical values of $h_j$ are exactly those of the maps $g_j^i$. 
         By~\ref{item:criticalvalues}, the critical values of $g_j^i$ are in 
          $f_0(\Delta_j^i)\subset f_0(D_j) = D_{j+1}$, establishing~\ref{item:omegainDj}. 
          Furthermore, by~\ref{item:twocriticalpoints}, 
           $g_j^i$ has exactly two different critical points,
          with different critical values belonging to $f_0(\Delta_j^i)$. Since different
          $\Delta_j^i$ are pairwise disjoint, and $f_0$ is injective, claim~\ref{item:injective} follows. 
          Finally,~\ref{item:criticalvaluesaway} is an immediate consequence of~\ref{item:criticalvalues} 
          and~\ref{item:annulusimage}.

        This completes the construction of the $h_j$.         By~\ref{item:continuousextension} and Royden's glueing lemma (or the quasiconformal removability
         of quasiarcs), the
          map $h_j$ is indeed a $K$-quasiregular map. It remains to establish~\ref{item:notinjective}--\ref{item:dilatationsupport}.
        \begin{enumerate}[(a)]
          \item Each $h^i_j$ has exactly two critical values,
            and $\# \Omega_0 = 1$. It follows inductively that $\# \Omega_j = 2^j$ for all $j$, and in particular each
             $h_j$ has at least one critical point. 
          \item If $c\in\mathcal{C}_j$ is a critical point of $h_j$, then $c = c_j^i$ for some $i$, by~\eqref{eqn:cji} and~\ref{item:injective}. So by~\ref{item:correctimageofcentre},  
           \[
		h_{j+1}(h_{j}(c)) = h_{j+1}(h_j(c_{j}^i)) = g_{j+1}^i(\omega_{j+1}^i) = c_{j}^i = c. \]
            \item   By definition of $V_j$ in~\eqref{eqn:Vjdefn}, $D_j\setminus V_j$ is compact. 
            \item $h_j|_{W_j}$ is injective by~\eqref{eqn:Wjdefn} and~\ref{item:annulusimage}. 
            \item By~\eqref{eqn:hjdefn}, $h_j=f_0$ on $V_j$. 
            \item By~\eqref{eqn:Vjdefn}, $V_{j+1}\supset h_j(W_j)$. 
            \item By~\ref{item:supportofdilatation}, the dilatation of $h_j$ is supported on $W_j$.\qedhere
        \end{enumerate}
       \end{proof}
         
   \begin{proof}[Proof of Theorems~\ref{thm:question5} and~\ref{thm:question4}]
    Choose $\mu$ as in Proposition~\ref{prop:moduli}, and let $(h_j)_{j=0}^{\infty}$ be as in Proposition~\ref{prop:hjquestion5}.
    Let $F$ be the quasiregular map defined by~\eqref{eqn:Fdefn}. Then $F^2(c)=c$ for every critical point $c$ of $F$. 

    The dilatation of $h_j$ is
     supported on $W_j$, and $h_j$ is conformal on $\overline{V_j}$. Moreover, $F^n(X_j)\subset V_{n+j}$ for all $n$ and $j$.     
     Hence $X_j \defeq W_j\setminus \overline{V_j}$ satisfies the conditions of Proposition~\ref{prop:qcsurgery}.
     Let $f$ be the  meromorphic function~\eqref{eqn:fdefn}; then $f$ has a Baker domain $U$ and satisfies $f^2(c)=c$ for every critical point of $f$. 
      This completes the proof of Theorem~\ref{thm:question5}.
     
     Furthermore, by Proposition~\ref{prop:moduli}, the function also satisfies~\ref{eqn:largeannuli} 
       for a sequence $R_j\to\infty$, proving Theorem~\ref{thm:question4}. 
   \end{proof}

\providecommand{\bysame}{\leavevmode\hbox to3em{\hrulefill}\thinspace}
\providecommand{\href}[2]{#2}

\end{document}